\documentclass[12pt,a4paper]{amsart}
\usepackage{amsmath,amssymb,amsthm,eucal,dsfont,bm,mathrsfs,stmaryrd}
\usepackage{amsfonts}
\usepackage{graphicx}
\usepackage{framed}
\usepackage{amssymb}
\usepackage{enumerate}

\usepackage[dvipsnames]{xcolor}

\definecolor{cornellred}{rgb}{0.7, 0.11, 0.11}

\usepackage[colorlinks=true,urlcolor=RoyalBlue,citecolor=RoyalBlue,linkcolor=RoyalBlue,linktocpage,pdfpagelabels,
bookmarksnumbered,bookmarksopen]{hyperref}

\usepackage{hyperref}

\newcommand{\B}[1]{\mathbb{#1}}
\newcommand{\C}[1]{\mathcal{#1}}
\newcommand{\vertiii}[1]{{\left\vert\kern-0.25ex\left\vert\kern-0.25ex\left\vert #1 
    \right\vert\kern-0.25ex\right\vert\kern-0.25ex\right\vert}}
\newcommand{\diam}[1]{diam(#1)}

\def\Xint#1{\mathchoice
   {\XXint\displaystyle\textstyle{#1}}%
   {\XXint\textstyle\scriptstyle{#1}}%
   {\XXint\scriptstyle\scriptscriptstyle{#1}}%
   {\XXint\scriptscriptstyle\scriptscriptstyle{#1}}%
   \!\int}
\def\XXint#1#2#3{{\setbox0=\hbox{$#1{#2#3}{\int}$}
     \vcenter{\hbox{$#2#3$}}\kern-.5\wd0}}

\def\dashint{\Xint-}

\usepackage{amsthm}
\usepackage{graphicx}
\usepackage{caption}
\usepackage{subcaption}
\theoremstyle{definition}
\newtheorem{definition}{Definition}[section]

\newtheorem{remark}[definition]{Remark}

\theoremstyle{plain}
\newtheorem{theorem}[definition]{Theorem}
\newtheorem{proposition}[definition]{Proposition}
\newtheorem{lemma}[definition]{Lemma}
\newtheorem{corollary}[definition]{Corollary}

\makeatletter
\@namedef{subjclassname@2020}{\textup{2020} Mathematics Subject Classification}
\makeatother

\numberwithin{equation}{section}
\begin{document}
\title[Morrey-Campanato Functional Spaces for Carnot Groups]{Morrey-Campanato Functional Spaces\\ for Carnot Groups}

\author[N.\,Cangiotti]{Nicol\`o Cangiotti}
\author[M.\,Capolli]{Marco Capolli}

\address[N.\,Cangiotti]{Department of Mathematics\newline\indent Politecnico di Milano \newline\indent
via Bonardi 9, Campus Leonardo, 20133 Milan, Italy}
\email{nicolo.cangiotti@polimi.it}

\address[M.\,Capolli - Previous affiliation]{Institute of Mathematics\newline\indent Polish Academy of Sciences\newline\indent Jana i Jedrzeja Sniadeckich 8, 00-656 Warsaw, Poland.}
\address[M.\,Capolli - Current affiliation]{Department of Mathematics ``Tullio Levi-Civita'' \newline\indent Università degli studi di Padova \newline\indent
via Trieste 63, 35121 Padua, Italy}
\email{marco.capolli@unipd.it}



\begin{abstract}
We shortly review the historical path of Morrey-Campanato functional spaces and the fundamentals of Carnot groups. Then, we merge these two topics, by recovering several classical results concerning regularity of Morrey-Campanato spaces in the framework of Carnot groups. 
\end{abstract}

\maketitle


\section{Introduction}
\label{Intro}

In the early 60s, Sergio Campanato published a series of works in the \emph{Annali della Scuola Normale Superiore di Pisa} \cite{Campanato61,Campanato63,Campanato64}. In these papers, he laid the foundations of a particular class of functional spaces that would become a landmark for a wide community of mathematician in the following years. However, despite the intrinsic elegance of Campanato's approach and its undoubted potential, the theory of elliptic and parabolic equations developed in the last decades seems to prefer the classical treatment based on the integral representation (as it was already noticed by Enrico Giusti in the 1978 \cite{Giusti78}). The purpose of this manuscript is then twofold. On the one hand, we want to embrace that part of literature that nowadays carries on the tradition descending from Campanato's original works. On the other hand, we shall provide a further research line by considering the generalization of such functional spaces in the framework of Carnot groups. In order to provide an exhaustive overview on the topic, it seems appropriate to shortly survey the main historical steps of this approach, highlighting also the current state of art of the research with most recent developments. 
\smallskip

In 1938, Charles B. Morrey published a manuscript \cite{Morrey38}, now well renowned, in which he studied the existence and regularity properties of the solutions of the so-called quasi-linear elliptic partial differential equations. In particular, he introduced the idea of approaching the regularity of solutions to PDE by considering the following heuristic inequality:
\[
\int_{B(x,r)} u(y)dy\le k\cdot r^{\lambda}, \quad \text{with } \lambda \ge 0,
\]
where $B(x,r)$ denotes the usual Euclidean ball of radius of radius $r$ centered in $x$ and $k$ is a suitable constant. In literature the spaces satisfying this kind of property are called $\mathcal{L}^{p,\lambda}$ spaces. In the following paragraph, we shall provide a formal description of a generalization of these function spaces. Similar techniques were developed in the following years by Louis Nierenberg \cite{Nirenberg53}, Robert Finn \cite{Finn53}, and Norman Meyer \cite{Meyer64}. However, the real turning point came in Italy during the 60s, with Enrico Giusti \cite{Giusti67}, Carlo Miranda \cite{Miranda63}, Guido Stampacchia \cite{Stampacchia60}, Giorgio Talenti \cite{Talenti65}, and especially with the already cited Sergio Campanato \cite{Campanato61, Campanato63, Campanato64}. Indeed, Campanato was the first\footnote{For this reason, the name Morrey-Campanato is often used when referring to those spaces.} to thoroughly investigate the space $\C{L}^{p,\lambda}$.

To be more specific, let us consider $\Omega$ a bounded and measurable subset of $\mathbb{R}^n$ with diameter $\rho_0$. Then, it is possible to fix two real numbers $p$ and $\lambda$, with $p\ge 1$ and $0 \le \lambda \le n$ and then define the functional space $\mathcal{L}^{p,\lambda}(\Omega)$ of the functions $u(x)$ defined in $\Omega$ such that 
\begin{equation}
\sup_{\substack{y\in \Omega\\\rho \in [0,\rho_0]}}\frac{1}{\rho^{\lambda}}\int_{B(y,\rho)\cap\Omega}|u(x)|^p<+\infty.
\end{equation}
By assuming as norm the following expression, we eventually get a complete Banach space:
\[
\vert\vert u \vert \vert_{\mathcal{L}^{p,\lambda}}(\Omega)=\left[ \sup_{\substack{y\in \Omega\\\rho \in [0,\rho_0]}}\frac{1}{\rho^{\lambda}}\int_{B(y,\rho)\cap\Omega}|u(x)|^p<+\infty\right]^{\frac{1}{p}}.
\]

Since their introduction, Morrey-Campanato spaces have represented a stimulating topic, whose interest is cyclically renewed decade by decade as testified by the different surveys regularly published through the years. We suggest, for the keen reader, the first survey due to Jaak Peetre \cite{Peetre69} dated 1969, the first update of 1979 written by Mitchell Taibleson and
Guido Weiss \cite{Taibleson80}, and the most recent review of Humberto Rafeiro, Natasha Samko and Stefan Samko \cite{Rafeiro2013}. As further evidence of the increasing appeal aroused by these kind of spaces, we remark that Morrey-Campanato spaces are at the heart of several recent studies. By way of illustration, we mention the work of David Adams and Jie Xiao, which analyzes several fundamental aspects of functional analysis and potential theory for the Morrey-Campanato
spaces in harmonic analysis \cite{Adams12} and the work of Dachun Yang, Dongyong Yang, and Yuan Zhou that applies localized Morrey-Camapanto space on Schr\"odinger operators \cite{yang10}. Other important contributions in this area are provided by the two volumes about the application of 
Morrey spaces on integral operators and PDE's, written by Yoshihiro Sawano, Giuseppe Di Fazio, Denny Ivanal Hakim and published in 2020 \cite{sawano2020}. Another thrilling book date-stamped 2010 is due to Wen Yuan, Winfried Sickel, and Dachun Yang \cite{Yuan10}, where the authors give a unified treatment of three kind of functional spaces, namely Besov and Triebel-Lizorkin spaces, Morrey-Campanato spaces and  the so-called $Q$ spaces (spaces of holomorphic functions on the unit disk). 

We finally recall a compelling application of Morrey spaces in the study of Navier-Stokes equations developed by Tosio Kato in 1992 \cite{Kato92}. This manuscript was forerunner of a series of works published in the following years as the book due to Hans Triebel \cite{Triebel13} and the article due to Pierre Gilles Lemarié-Rieusset \cite{Lemarie07}. We finally remark that Lemarié-Rieusset also studied the problem of pointwise multipliers between Morrey spaces, providing necessary and sufficient conditions for granting that the space of those multipliers is a Morrey space as well \cite{Lemarie13}.
\medskip

As mentioned above, Campanato provided a compelling generalization, in which are included (by varying some parameters) not only Morrey spaces, but also H\"older spaces \cite{Campanato64}. A result in \cite{Campanato64} that is key through the paper is a lemma which is attributed by Campanato to Ennio de Giorgi. Such lemma gives an estimate on the norm of the derivative of a polynomial based on the integral of the whole polynomial (see \ref{degiorgieuclideo} for the details). In a recent paper by Marco Capolli, Andrea Pinamonti, and Gareth Speight \cite{Capolli23}, an analogous of de Giorgi's lemma was proved, for the case $p=1$, in the setting of Carnot groups. Inspired by this, we decided to merge Campanato's work with the theory of Carnot groups. Thus, our study can be included in the recent research line, which is focused on this kind of combination of mathematical tools. For instance, in 2010, Vagif Guliyev and Ali Akbulut Yagub Mammadov proved the boundedness of the fractional maximal operator and their commutators on the Carnot group in generalized Morrey spaces \cite{Guliyev13}. A few year later, continuing  along the same lines, Ahmet Eroglu, Vagif Guliyev, and Javanshir Azizov proved two Sobolev-Stein embedding theorems on generalized Morrey spaces in the Carnot group setting \cite{Eroglu17}. Moreover, in 2013, Kabe Moen and Virginia Naibo introduced Leibniz type rules in Morrey-Campanato spaces by introducing the notions of higher-order weighted multilinear Poincaré and Sobolev
inequalities in Carnot groups \cite{Moen14}.
\medskip

The manuscript is structured as follows. Section \ref{Sec2} shortly reviews the basis of Carnot groups theory, which represent the main framework in this paper. In Section \ref{Sec3}, we shall introduce the class of function proposed by Campanato, by replacing the usual Euclidean space with a generic Carnot group. Section \ref{Sec4} is devoted to present our main results about the regularity depending on the parameters introduced in the previous sections. Finally, in Section \ref{Sec5} we take stock of our work, suggesting possible directions for future research lines. 
\bigskip

\section{Preliminaries on Carnot groups}
\label{Sec2}

For the sake of completeness, we recall in this section some preliminaries on Carnot groups, underlying the basic concepts and definitions as well as some fundamental results.

A \emph{Lie group} $\B{G}$ is a smooth manifold which is also a group and for which the group operations are smooth maps. The \emph{Lie algebra} $\mathfrak{g}$ associated to a Lie group is the space of left invariant vector fields equipped with the Lie bracket $[\cdot, \cdot]\colon \mathfrak{g}\times\mathfrak{g}\to \mathfrak{g}$ where
\[[X,Y](f)=X(Y(f))-Y(X(f)),\quad \mbox{for }f\colon \B{G}\to \mathbb{R}\ \mbox{smooth}.\]

\begin{definition}\label{Carnotdef}
A simply connected Lie group $\B{G}$ is said to be a \emph{Carnot group of step $s$} if there exist linear subspaces $V_1, \dots ,V_s$ of $\mathfrak{g}$ such that
\[\mathfrak{g}=V_1\oplus \dots \oplus V_s\]
with
\[[V_1,V_{i}]:=\mathrm{span}\{[a,b]: a\in V_1,\ b\in V_i\}=V_{i+1} \mbox{ if }1\leq i\leq s-1\]
and
$[V_1,V_s]=\{0\}$.

$V_1\oplus \dots \oplus V_s$ is called \emph{stratification} of $\mathfrak{g}$.

Let $m_i:=\dim(V_i)$ and define $h_i:=m_1+\dots +m_i$ for $1\leq i\leq s$. We also call $N:=h_s=\dim{\mathfrak{g}}$. A basis $X_1,\dots, X_N$ of $\mathfrak{g}$ is said to be \emph{adapted to the stratification} if $X_{h_{i-1}+1},\dots, X_{h_{i}}$ is a basis of $V_i$ for $1\leq i \leq s$. 
\end{definition}

For the remainder of this paper, $\B{G}$ will be a Carnot group of step $s$ with Lie algebra $\mathfrak{g}$ admitting a stratification as in Definition \ref{Carnotdef}. We will also consider to have fixed a basis adapted to the stratification of $\B{G}$.

The map $\exp \colon \mathfrak{g}\to \mathbb{G}$ is defined by $\exp(X)=\gamma(1)$, where $\gamma \colon [0,1] \to \mathbb{G}$ is the unique solution to the problem
\[
\begin{cases} \gamma'(t)=X(\gamma(t))\\ \gamma(0)=e\end{cases}
\]
where $e$ is the identity element of $\mathbb{G}$.  This \emph{exponential map} is a diffeomorphism between $\mathbb{G}$ and $\mathfrak{g}$. We will identify $\mathbb{G}$ with $\mathbb{R}^{N}$ by:
\[ \exp(x_{1}X_{1}+\dots +x_{N}X_{N})\in \mathbb{G} \longleftrightarrow (x_{1}, \dots, x_{N})\in \mathbb{R}^{N}.\]
With this identification, we denote points of $\mathbb{G}$ by $(x_{1}, \dots, x_{N})\in \mathbb{R}^{N}$.
\begin{definition}\label{def_homog}
    The \emph{homogeneity} $d_i\in\B{N}$ of the coordinate $x_i$ is defined by
    \[ d_i:=j \quad\text {whenever}\quad h_{j-1}+1\leq i\leq h_{j}.\]
\end{definition} 
Thanks to the definition of homogeneity we can define dilations: for any $\lambda >0$, the \emph{dilation} $\delta_\lambda\colon \B{G}\to\B{G}$, is defined in coordinates by
    \[\delta_\lambda(x_1, \dots ,x_N)=(\lambda^{d_1}x_1, \dots ,\lambda^{d_N}x_N).\]
Dilations satisfy $\delta_{\lambda}(xy)=\delta_{\lambda}(x)\delta_{\lambda}(y)$ and $\left(\delta_{\lambda}(x)\right)^{-1}=\delta_{\lambda}(x^{-1})$ where $x^{-1}$ denotes the inverse element of $x$.

A \emph{Haar measure} $\mu$ on $\mathbb{G}$ is a non-trivial Borel measure on $\mathbb{G}$ so that $\mu(gE)~=~\mu(E)$ for any $g\in \mathbb{G}$ and Borel set $E\subset \mathbb{G}$. Any constant multiple of the $N$ dimensional Lebesgue measure $\mathcal{L}^{N}$ is a Haar measure for $\B{G}$. 

\begin{definition}\label{horizontalcurve}
An absolutely continuous curve $\gamma\colon [a,b]\to \mathbb{G}$ is \emph{horizontal} if there exist $u_{1}, \dots, u_{N}\in L^{1}[a,b]$ such that $\gamma'(t)=\sum_{j=1}^{N}u_{j}(t)X_{j}(\gamma(t))$ for almost every $t\in [a,b]$.

We also define the \emph{(horizontal) length} of $\gamma$ as $$
L(\gamma):=\int_{a}^{b}|u(t)|dt,
$$ 
where $u=(u_{1}, \dots, u_{N})$ and $|\cdot|$ denotes the Euclidean norm on $\mathbb{R}^{N}$.
\end{definition}

Thanks to the Chow-Rashevskii Theorem \cite[Theorem 9.1.3]{BLU07} we know that any two points in $\mathbb{G}$ can be connected by horizontal curves. This allows us to define the \emph{Carnot-Carath\'eodory distance (CC distance)} as
\[
d(x,y):=\inf \{ L(\gamma) : \gamma \colon [0,1]\to \mathbb{G} \mbox{ horizontal joining }x\mbox{ to }y \}.
\]
As customary, we will denote the CC ball with center $x$ and radius $r$ by $B(x,r)$. It can be proved that, by defining $Q:=\sum_{i=1}^s i m_i$ (recall $s$ and $m_i$ from Definition \ref{Carnotdef}), the following holds
\[
\mathcal{L}^N(B(x,r))=r^Q \mathcal{L}^N(B(0,1)).
\]
Notice that $Q\geq N$ and the equality holds only when $\B{G}$ is exactly $\B{R}^N$.

The CC distance induces on $\mathbb{G}$ the same topology as the Euclidean distance, so terms like open, closed, and compact sets can be used without ambiguity.

The CC distance is not bi-Lipschitz equivalent to the Euclidean distance, however it is possible to prove the following lemma (see \cite[Proposition 5.15.1]{BLU07})

\begin{lemma}\label{ballequivalence}
    Given $K\subset \mathbb{G}$ compact, there exists a constant $c>0$, depending only on $K$, such that
    \begin{equation}
        \frac{1}{c}|x-y|\leq d(x,y)\leq c|x-y|^{\frac{1}{s}}\qquad \mbox{for all } x,y\in K
    \end{equation}
where $|\cdot|$ denotes the Euclidean norm and points $x$ and $y$ are intended either as points in $\B{G}$ or in $\B{R}^N$, based on the context.
\end{lemma}

\subsection{Polynomials and Smooth Functions}
In this paragraph, we present a brief introduction to Polynomials and smooth functions on Carnot groups. We follow the notation introduced in \cite{FS}. 

A \emph{multi-index} $J=(j_1,\dots, j_N)$ is an ordered list of $N$ non-negative integers. Given a multi-index, we define its \emph{norm} as $|J| = \sum_{i=1}^N j_i$ and its \emph{homogeneous norm} as $|J|_{\B{G}} = \sum_{i=1}^N d_i j_i$.  We also set $J!:=j_{1}!j_{2}!\dots j_{N}!$. The two norms are related by the inequality:
\begin{equation}\label{normindexequivalence}
    |J|\leq |J|_{\B{G}}\leq d_N |J|
\end{equation}

Whether $|\cdot|$ is used for the Euclidean norm or for the norm of a multi-index will be clear from the context.

We define homogeneous polynomials as follows, see \cite{PV06, BLU07}.

\begin{definition}
Given coordinates $x=(x_1,\dots,x_N)$ in $\mathbb{R}^N$, we define:
\begin{itemize}
\item A \emph{monomial of homogeneous degree $d\geq 0$} is a polynomial of the form $x^J:=x_1^{j_1}\dots x_N^{j_N}$ for some multi-index $J=(j_1,\dots, j_N)$ with $|J|_{\mathbb{G}}=d$.
\item A \emph{homogeneous polynomial of homogeneous degree $d$} is a linear combination of monomials of the same homogeneous degree $d$. 
\item A \emph{polynomial of homogeneous degree at most $d$} is a linear combination of monomials of homogeneous degree at most $d$.
\end{itemize}
\end{definition}

Given a multi-index $I$, we denote higher order derivatives by $X^I=X_1^{i_1}\dots X_N^{i_N}$ and $\left( \frac{\partial}{\partial x} \right)^{I}=\frac{\partial^{i_{1}}}{\partial x_{1}^{i_{1}}}\dots \frac{\partial^{i_{N}}}{\partial x_{1}^{i_{N}}}$. The following lemma is \cite[Proposition 20.1.5]{BLU07}.

\begin{lemma}\label{multiderivative}
For every multi-index $I$ we have
\[
X^I=\left( \frac{\partial}{\partial x} \right)^I + \sum_{\substack{J \neq I \\ |J| \leq |I|\\ |J|_{\B{G}}\geq |I|_{\B{G}}}} Q_{J,I} \left(\frac{\partial}{\partial x}\right)^J
\]
where the $Q_{I,J}$ are homogeneous polynomials of homogeneous degree $$
|J|_{\B{G}} - |I|_{\B{G}}.
$$
This equation is meant in the sense that both left and right define differential operators on $\B{R}^{N}$ whose action agrees on $C^{\infty}$ functions on $\B{R}^{N}$.
\end{lemma}

The following lemma is a direct consequence of the results presented so far.

\begin{lemma}\label{lemmaderivative}
Suppose $P$ is a homogeneous polynomial of homogeneous degree $k$ and $I$ is a multi-index. Then $X^I P$ is either identically zero, or it is a homogeneous polynomial of homogeneous degree
$$
k-|I|_{\B{G}}.
$$
In the particular case $|I|_{\B{G}}=k$, $X^I P$ is a constant. 
\end{lemma}

\begin{definition}
Given an open set $\Omega\subset \mathbb{G}$ and $k\in\mathbb{N}$, we define
\[
C^k_{\mathbb{G}}(\overline{\Omega}):=\left\{u: \Omega\to \mathbb{R} \colon X^I u \mbox{ exists and is continuous for all }  |I|_{\B{G}}\leq k\right\}.
\]
Similarly we define, for $0<\alpha\leq 1$, $C^{k,\alpha}_{\B{G}}(\overline{\Omega})$ as the subspace of $C^k_{\B{G}}(\overline{\Omega})$ of functions whose derivative not only exist continuous, but are also H\"olderian in $\overline{\Omega}$ with exponent $\alpha$.
\end{definition}
To avoid a heavy notation, we will omit the $\B{G}$ in $C^{k}_{\B{G}}(\overline{\Omega})$ and $C^{k,\alpha}_{\B{G}}(\overline{\Omega})$.
\medskip

It is known that $C^k(\overline{\Omega})$ and $C^{k,\alpha}(\overline{\Omega})$ are (complete) Banach spaces with respect to the norms
\[
[u]_{k,\Omega} = \sum_{|I|_{\B{G}}\leq k}\sup_{\overline{\Omega}}|X^I u|
\]
and
\begin{equation}\label{normCkalpha}
    [u]_{k,\alpha,\Omega} = [u]_{k,\Omega} + \sup_{|I|_{\B{G}} = k} \left\{\sup_{\stackrel{x,y\in\Omega}{x\neq y}}\frac{|X^I u(x) - X^I u(y)|}{d(x,y)^{\alpha}}\right\}
\end{equation}

The following lemma is attributed to De Giorgi by Campanato in \cite{Campanato64}. 

\begin{lemma}[De Giorgi]\label{degiorgieuclideo}
Let $E\subset B(x_0,r)\subset\B{R}^n$ be measurable and such that $\mathcal{L}^n(E)\geq Ar^n$ for some constant $A>0$. Then, for each positive integer $k$ and for each $p\geq 1$ there exists a positive constant $C$, depending only on $n,k,p$ and $A$, such that
\[
\left|\left[\left(\frac{\partial}{\partial x}\right)^I P(x)\right]_{x=x_0}\right|^p\leq \frac{C}{r^{n+p|I|}}\int_E |P(x)|^p d x
\]
for all polynomials $P$ of degree at most $k$ and for all multi-indices $I$.
\end{lemma}

Since this lemma is crucial in \cite{Campanato64}, we prove an analogue in Carnot groups.

\begin{lemma}[De Giorgi Lemma for Carnot groups]\label{degiorgicarnot}
Let $E\subset B(x_0,r)\subset \B{G}$ be measurable and such that $\mathcal{L}^N(E)\geq Ar^Q$ for some constant $A>0$. 

Then, for each positive integer $k$ and for each $p\geq 1$ there exists a positive constant $C$, depending only on $Q,k,p$ and $A$, such that
\[
|[X^{I}P(x)]_{x=x_0}|^p\leq \frac{C}{r^{Q+p|I|_{\B{G}}}}\int_E |P(x)|^p dx
\]
for all polynomials $P$ of homogeneous degree at most $k$ and multi-indices $I$.
\end{lemma}

\begin{proof}[Sketch of Proof]
Notice that the proof for the case $p=1$ is also present in \cite{Capolli23}.

We first prove the lemma for $r=1$ and $x_{0}=0$. From Lemma \ref{multiderivative} we have, for each multi-index $I$ with $|I|_{\B{G}}<k+1$, 
\begin{equation}\label{firstequation}
|[X^I P(x)]_{x=0}|^p \leq \sum_{\substack{|J|\leq |I|\\ |J|_{\B{G}}\geq |I|_{\B{G}}}} |Q_{J,I}(0)|^p \left|\left[\left(\frac{\partial}{\partial x}\right )^J P(x) \right]_{x=0}\right|^p.
\end{equation}
By Lemma \ref{ballequivalence}, $B(0,1)\subset B_\mathbb{E}(0,c)$, where $B_\mathbb{E}(0,c)$ denotes the Euclidean ball centred at $0$ with radius $c$. Moreover, thanks to \eqref{normindexequivalence}, any polynomial of homogeneous degree at most $k$ has also standard degree at most $k$. We can then use Lemma \ref{degiorgieuclideo} to estimate each terms in the right hand side of \eqref{firstequation}. We get
\begin{equation}
\label{secondequation}
\eqref{firstequation} \leq C_{1} \sum_{\substack{|J|\leq |I|\\ |J|_{\B{G}}\geq |I|_{\B{G}}}} |Q_{J,I}(0)|^p \int_{E} |P(x)|^p dx
\end{equation}
For an appropriate constant $C_1$, which proves the lemma in the first case.

To prove the general case, consider the map $T:\mathbb{G}\to\mathbb{G}$ defined as $T(x):=\delta_{\frac{1}{r}}(x_0^{-1}x)$. Hence $T^{-1}(x)=x_0\delta_{r}x$ and the classical change of variable formula gives:
\begin{equation}\label{cambio}
\int_{E}|P(x)|^p dx = r^Q\int_{T(E)}|P(x_0\delta_{r} x)|^p dx.
\end{equation}
Define $S(y):=P(x_0\delta_r y)$. Then, $S(y)$ is a  is a homogeneous polynomial of homogeneous degree at most $k$ in the variable $y$. Clearly $T(E)\subset B(0,1)$ and $\mathcal{L}^N(T(E))\geq \frac{1}{r^Q}\mathcal{L}^N(E)\geq A.$ Hence, we can use \eqref{secondequation} with $E$ replaced by $T(E)$, $P$ replaced by $S$ and \eqref{cambio} to get
\begin{align}
|[X^{I} S(y)]_{y=0}|^p\leq C \int_{T(E)} |S(y)|^p\, dy=\frac{C}{r^Q}\int_E |P(y)|^p dy
\end{align}
for an appropriate constant $C$. The proof is concluded by showing that 
\begin{equation}\label{tesifinaledegiorgi}
[X^{I} S(y)]_{y=0}=r^{|I|_{\mathbb{G}}}[X^{I}P(x)]_{x=x_{0}} \qquad \mbox{for every multi-index }I.
\end{equation}
We remand to \cite{Capolli23} for a complete version of the proof.
\end{proof}


\section{The class $\C{L}_k^{p,\lambda}(\Omega)$}
\label{Sec3}

In this section we will follow the original argument of Campanato \cite{Campanato64} by introducing the class $\C{L}_k^{p,\lambda}(\Omega)$. We will then proceed to prove some key properties of functions belonging to this class.

From now on, $\Omega$ will be an open, connected and bounded subset of $\B{G}$. We will denote, as it is customary, the frontier of $\Omega$ with $\partial\Omega$ and its closure with $\overline{\Omega}$. Moreover, for any $x_0\in\B{G}$ and $r>0$ we will define $\Omega(x_0,r):=\Omega\cap B(x_0,r)$.

\begin{definition}[T condition]
    We say that $\Omega$ satisfies the \emph{thick condition} if there exists a positive constant $A$ such that, for each $x_0\in\overline{\Omega}$ and for each $r>0$, we have
    \begin{equation}
    \C{L}^Q(\Omega(x_0,r))\geq Ar^Q.\tag{T}\label{thin}
    \end{equation}
\end{definition}

\begin{definition}
Let $u\in L^p(\Omega)$. We say that $u$ belongs to the class $\C{L}^{p,\lambda}_k(\Omega)$ if 
\begin{equation}\label{seminorm}
   \vertiii{u}_{k,p,\lambda}:=\sup_{\substack{x_0\in\overline{\Omega} \\ r>0}}\left[ \frac{1}{r^{\lambda}}\inf_{P\in\C{P}_k} \int_{\Omega(x_0,r)} |u(x) - P(x)|^p dx\right]^{\frac{1}{p}}<+\infty.
\end{equation}
\end{definition}

It is easy to see that $\vertiii{\cdot}_{k,p,\lambda}$ is a seminorm in $\C{L}_k^{p,\lambda}(\Omega)$. It can be completed to a norm $\|\cdot\|_{k,p,\lambda}$ in the standard way by defining
\[
\|u\|_{k,p,\lambda} := \left[ \|u\|^p_{L^p} + \vertiii{u}_{k,p,\lambda}^p \right]^{\frac{1}{p}}.
\]

We now prove a result about the realization of the $\inf (\cdot)$ in the definition of $\vertiii{u}_{k,p,\lambda}$.

\begin{proposition}\label{inf_realization}
Let $u\in\C{L}_k^{p,\lambda}(\Omega)$. Then for each $x_0\in\overline{\Omega}$ and $r>0$ there exists an unique polynomial $P_k(x,x_0,r,u)$ such that
\[
\inf_{P\in\C{P}_k} \int_{\Omega(x_0,r)} |u(x) - P(x)|^p dx = \int_{\Omega(x_0,r)} |u(x) - P_k(x,x_0,r,u)|^p dx.
\]
\end{proposition}

\begin{proof}
Fix $x_0$ and $r$ as in the hypothesis. Every polynomial in $\C{P}_k$ can be written, up to a change of variable, in the form
\[
P(x) = \sum_{|I|_{\B{G}}\leq k} \frac{a_I(x_0)}{I!}(x_0^{-1}x)^I.
\]
By rearranging the coefficients $a_I$, we can see every polynomial in $\C{P}_k$ as uniquely determined by a point in an Euclidean space $\B{R}^N$ for an appropriate $N$. The difference $u(x)-P(x)$ is in $L^p(\Omega(x_0,r))$, hence its $L^p$ norm is a positive and continuous function that depends only on the coefficients of $P(x)$, let us call it $f:\B{R}^N\to\B{R}_+$. The proposition will be proved by showing that the infimum of $f$ in $\B{R}^N$ exists and is unique.
\medskip

\noindent
\textit{\underline{Claim}:} We can search the infimum in a compact that contains the origin.
\medskip

\noindent
\textit{Proof of Claim.} We have the relation
\[
f(y)\geq \left| \|u\|_{L^p} - \|P(x)\|_{L^p} \right|^p
\]
where $P(x)$ is the polynomial determined by the point $y=(a_i)_{i=1}^N\in\B{R}^N$ as explained above. Notice that $\|u\|_{L^p}$ is fixed and does not depend on $y$. From \ref{degiorgicarnot} we get that $\|P(x)\|_{L^p}\geq \tilde{C} a_i $ for each $i=1,\dots,N$, where we collected the constants from the Lemma in $\tilde{C}$. Hence $f(y)\to+\infty$ whenever 
\[
|y|=\sqrt{\sum_{i=1}^N a_i^2}\to+\infty,
\]
proving the claim.

\vspace{0.6em}

The existence of the infimum for $f$ follows from the fact that we restricted the search to a compact set, while the uniqueness follows from the uniform convexity of $L^p$ spaces.
\end{proof}

\begin{definition}\label{def_aI}
For any fixed $u\in\C{L}^{p,\lambda}_k(\Omega)$, we define
\begin{equation*}
    a_I(x_0,r,u):=[X^I(P_k(x,x_0,r,u))]_{x=x_0}.
\end{equation*}
Since the polynomial $P_k(x,x_0,r,u)$ is uniquely determined, we will simply write $P_k(x,x_0,r)$ and $a_I(x_0,r)$ when the $u$ is clear from the context.
\end{definition}

The following lemma gives an estimate of the behaviour of the polynomial $P_k(x,x_0,r)$ when the radius $r$ changes in a controlled way.

\begin{lemma}\label{lemma_concentric}
    Let $u\in \C{L}_k^{p,\lambda}(\Omega)$, then there exists a constant $K=K(p,\lambda)$ such that, for each $x_0\in\overline{\Omega}$, $r>0$ and $h\in\B{N}$ we have
    \[
    \int_{\Omega(x_0,\frac{r}{2^{h+1}})}\left|P_k\left(x,x_0,\frac{r}{2^h}\right)-P_k\left(x,x_0,\frac{r}{2^{h+1}}\right)\right|^pdx\leq K\vertiii{u}^p_{k,p,\lambda}\left(\frac{r}{2^h}\right)^{\lambda}.
    \]
\end{lemma}

\begin{proof}
    Let $x_0,r$ and $h$ be as in the hypothesis. For $x\in\Omega(x_0,\frac{r}{2^{h+1}})$, by using Jensen's inequality, we have
    \begin{align*}
        \left|P_k\left(x,x_0,\frac{r}{2^h}\right)-P_k\left(x,x_0,\frac{r}{2^{h+1}}\right)\right|^p & \leq 2^{p-1}\left|P_k\left(x,x_0,\frac{r}{2^h}\right)-u(x)\right|^p\\
        & + 2^{p-1}\left|P_k\left(x,x_0,\frac{r}{2^{h+1}}\right)-u(x)\right|^p.
    \end{align*}
    By integrating over $\Omega(x_0,\frac{r}{2^{h+1}})$ and observing that $\Omega(x_0,\frac{r}{2^{h+1}})\subset\Omega(x_0,\frac{r}{2^{h}})$ we get
    \begin{equation*}
    \begin{split}
        \int_{\Omega(x_0,\frac{r}{2^{h+1}})}&\left|P_k\left(x,x_0,\frac{r}{2^h}\right)-P_k\left(x,x_0,\frac{r}{2^{h+1}}\right)\right|^p dx\\
        &\leq 2^{p-1} \int_{\Omega(x_0,\frac{r}{2^{h}})}\left|P_k\left(x,x_0,\frac{r}{2^h}\right)-u(x)\right|^pdx \\
        & + 2^{p-1}\int_{\Omega(x_0,\frac{r}{2^{h+1}})}\left|P_k\left(x,x_0,\frac{r}{2^{h+1}}\right)-u(x)\right|^p dx \\
        & \stackrel{\dagger}{\leq} 2^{p-1}\left[\left(\frac{r}{2^{h}}\right)^{\lambda}\vertiii{u}_{k,p,\lambda}^p + \left(\frac{r}{2^{(h+1)}}\right)^{\lambda}\vertiii{u}_{k,p,\lambda}^p\right]\\
        & = 2^{p-1}(1+2^{-\lambda})\left(\frac{r}{2^{h}}\right)^{\lambda}\vertiii{u}_{k,p,\lambda}^p
    \end{split}
    \end{equation*}
    where $\dagger$ comes from the definition of $\vertiii{u}_{k,\lambda,p}$ and Proposition \ref{inf_realization}. The proof is concluded by choosing the constant $K(p,\lambda):=2^{p-1}(1+2^{-\lambda})$.
\end{proof}

We now prove two results about how $a_I(x_0,r)$ behaves with respect to changes in the base point $x_0$ or in the radius $r$.

\begin{lemma}\label{lemma_changebasepoint}
    Suppose that $\Omega$ has the \eqref{thin} condition and $u\in \C{L}^{p,\lambda}_k(\Omega)$. Then for each pair of points $x_0,y_0\in\overline{\Omega}$ and for each multi-index $I$ with $|I|_{\B{G}}=k$ we have
        \[
        |a_I(x_0,2\rho)-a_I(y_0,2\rho)|^p\leq C 2^{p+\lambda} \vertiii{u}^p_{k,p,\lambda}\rho^{\lambda-Q-pk}
        \]
        where $\rho = |y_0^{-1}x_0|$ and $C$ is the constant from Lemma \ref{degiorgicarnot}.
\end{lemma}

\begin{proof}
    Take $x_0,y_0$ as in the hypothesis. Define $\Omega(x_0,y_0):=\Omega(x_0,2\rho)\cap\Omega(y_0,2\rho)$, where $\rho=|y_0^{-1}x_0|$. For almost all $x\in \Omega(x_0,y_0)$ we have
    \begin{equation}\label{eq_2eta}
    \begin{split}
    |P_k(x,x_0,2\rho)-P_k(x,y_0,2\rho)|^p&\leq 2^{p-1}\left[|P_k(x,x_0,2\rho)-u(x)|^p\right.\\
    &\left.+|P_k(x,y_0,2\rho)-u(x)|^p\right].
    \end{split}
    \end{equation}
    We now observe that, by how $\rho$ was defined, $\Omega(x_0,\rho)\subset\Omega(y_0,2\rho)$. Clearly also $\Omega(x_0,\rho)\subset\Omega(x_0,2\rho)$ holds, hence $\Omega(x_0,\rho)\subset\Omega(x_0,y_0)$. We can then integrate \ref{eq_2eta} over $\Omega(x_0,\rho)$ and get
    \begin{equation}\label{stima1}
    \begin{split}
    \int_{\Omega(x_0,\rho)}|P_k(x,x_0,2\rho)-P_k(x,y_0,2\rho)|^p dx &\leq 2^{p-1}\left[\int_{\Omega(x_0,2\rho)}|P_k(x,x_0,2\rho)-u(x)|^p dx\right.\\
    &\left.+\int_{\Omega(x_0,2\rho)}|P_k(x,y_0,2\rho)-u(x)|^p dx\right]\\
    &\leq 2^{p+\lambda}\vertiii{u}^p_{k,p,\lambda}\rho^{\lambda}.
    \end{split}
    \end{equation}
    In order to complete the proof of the lemma we have to estimate from below the left-hand side of the previous inequality. To do so, we first define the polynomial $R(x):=P_k(x,x_0,2\rho)-P_k(x,y_0,2\rho)$. Notice that it is a polynomial in $x$ of homogeneous degree $k$. We can then apply Lemma \ref{degiorgicarnot} to it, with $E=\Omega(x_0,\rho)$ (recall that we assumed the \eqref{thin} condition for $\Omega$). We get
    \begin{equation}\label{stima2}
    \begin{split}
    |[X^{I}R(x)]_{x=x_0}|^p & \leq \frac{C}{\rho^{Q+p|I|_{\B{G}}}}\int_{\Omega(x_0,\rho)} |R(x)|^p dx\\
    & = \frac{C}{\rho^{Q+p|I|_{\B{G}}}}\int_{\Omega(x_0,\rho)}|P_k(x,x_0,2\rho)-P_k(x,y_0,2\rho)|^p dx.
    \end{split}
    \end{equation}
    When we restrict to the particular case of a multi-index $I$ with $|I|_{\B{G}}=k$, as in the hypothesis of the lemma, we get
    \begin{equation}\label{stima3}
    \begin{split}
    |[X^{I}R(x)]_{x=x_0}|^p &= |[X^I P_k(x,x_0,2\rho) - X^I P_k(x,y_0,2\rho)]_{x=x_0}|^p\\
    &= |a_I(x_0,2\rho) - a_I(y_0,2\rho)|^p
    \end{split}
    \end{equation}
    where the last inequality follows from Lemma \ref{lemmaderivative} since constants does not depend on the point in which we are calculating $X^I$ (recall that $R(x)$ has homogeneous degree $k$). The proof is then concluded by combining \eqref{stima1}, \eqref{stima2} and \eqref{stima3}.
\end{proof}

\begin{lemma}\label{lemma_changeradius}
    Suppose that $\Omega$ has the \eqref{thin} condition and $u\in \C{L}^{p,\lambda}_k(\Omega)$. Then for each $x_0\in\overline{\Omega}$, $r>0$, $h\in\B{N}$ and multi-index $I$ with $|I|_{\B{G}}\leq k$ we have
    \[
    \left|a_I(x_0,r)-a_I\left(x_0,\frac{r}{2^h}\right)\right|\leq M \vertiii{u}_{k,p,\lambda} r^{\frac{\lambda-Q-p|I|_{\B{G}}}{p}} \sum_{j=0}^{h-1} 2^{j\frac{Q+p|I|_{\B{G}}-\lambda}{p}}
    \]
    where $M$ is a constant that does not depends on $x_0$, $r$, $h$ and $I$. 
\end{lemma}

\begin{proof}
    Let $x_0,r,h$ and $I$ be as in the hypotheses. From the triangular inequality and from the definition of $a_I(x_0,r)$ we get
    \begin{equation}\label{telescopica}
        \begin{split}
        \left|a_I(x_0,r)-a_I\left(x_0,\frac{r}{2^h}\right)\right| &\leq \sum_{j=0}^{h-1}\left|a_I\left(x_0,\frac{r}{2^j}\right)-a_I\left(x_0,\frac{r}{2^{j+1}}\right)\right|\\
        &\leq \sum_{j=0}^{h-1}\left| \left[ X^I\left( P_k\left(x,x_0,\frac{r}{2^j}\right) - P_k\left( x,x_0,\frac{r}{2^{j+1}} \right) \right)\right]_{x=x_0} \right|.
        \end{split}
    \end{equation}
    We now apply Lemma \ref{degiorgicarnot} to the polynomial $P_k\left(x,x_0,\frac{r}{2^j}\right) - P_k\left( x,x_0,\frac{r}{2^{j+1}} \right)$ to get
    \begin{equation}\label{telescopica2}
        \begin{split}
        &\left|a_I(x_0,r)-a_I\left(x_0,\frac{r}{2^h}\right)\right| \\
        &\leq C^{\frac{1}{p}} \left(\frac{2}{r}\right)^{\frac{Q}{p}+|I|_{\B{G}}}\sum_{j=0}^{h-1} 2^{j(\frac{Q}{p}+|I|_{\B{G}})} \left[\int_{\Omega(x_0,\frac{r}{2^{j+1}})} \left|P_k\left(x,x_0,\frac{r}{2^j}\right) - P_k\left(x,x_0,\frac{r}{2^{j+1}} \right)\right|^p dx\right]^{\frac{1}{p}}.
        \end{split}
    \end{equation}
    By applying Lemma \ref{lemma_concentric} to each integral in the right-hand side of the above inequality we finally get
    \begin{equation*}
        \begin{split}
        \left|a_I(x_0,r)-a_I\left(x_0,\frac{r}{2^h}\right)\right| &\leq C^{\frac{1}{p}} \left(\frac{2}{r}\right)^{\frac{Q}{p}+|I|_{\B{G}}} \sum_{j=0}^{h-1} 2^{j(\frac{Q}{p}+|I|_{\B{G}})} K^{\frac{1}{p}} \vertiii{u}_{k,p,\lambda} \left( \frac{r}{2^j} \right)^{\frac{\lambda}{p}}\\
        &= (CK)^{\frac{1}{p}} 2^{\frac{Q}{p}+|I|_{\B{G}}} \vertiii{u}_{k,p,\lambda} r^{\frac{\lambda-Q-p|I|_{\B{G}}}{p}} \sum_{j=0}^{h-1} 2^{j\frac{Q+p|I|_{\B{G}}-\lambda}{p}}.
        \end{split}
    \end{equation*}
    The proof is concluded by observing that $2^{|I|_{\B{G}}}\leq 2^k$ and choosing the constant $M(Q,k,p,A,\lambda) = (CK2^{Q+pk})^{\frac{1}{p}}$.
\end{proof}

The following lemma is the key result of this section as it introduces a family of functions that will be the main subject of the results in the next section.

\begin{lemma}\label{lemmavI}
    Suppose $\Omega$ has the \eqref{thin} condition and $u\in \C{L}^{p,\lambda}_k(\Omega)$ with $\lambda>Q+p\alpha$ for a given natural number $\alpha\leq k$. Then there exists a family of real-valued functions $\{v_I(x)\}_{|I|_{\B{G}}\leq \alpha}$, each defined on $\overline{\Omega}$, such that for each $x_0\in\overline{\Omega}$, $r>0$ and multi-index $I$ with $|I|_{\B{G}}\leq\alpha$ we have
    \begin{equation}\label{stimavI}
    |a_I(x_0,r)-v_I(x_0)|\leq N \vertiii{u}_{k,p,\lambda} r^{\frac{\lambda - Q - p|I|_{\B{G}}}{p}}
    \end{equation}
    where $N$ is a constant that does not depend on $x_0, r, \alpha$ and $I$.
\end{lemma}

\begin{proof}
    Let $x_0,r,\alpha$ and $I$ be as in the hypotheses. We want to prove that, independently from the choice of $r$, the sequence $\{a_I\left(x_0,\frac{r}{2^h}\right)\}$ converges for $h\to\infty$. To do so, suppose that $i,j\in\B{N}$ with $i<j$. We apply Lemma \ref{lemma_changeradius} with $r'=\frac{r}{2^i}$ to get
    \begin{equation}\label{eqseries}
    \begin{split}
        \left|a_I\left(x_0,\frac{r}{2^j}\right)-a_I\left(x_0,\frac{r}{2^i}\right)\right| &= \left|a_I\left(x_0,r'\right)-a_I\left(x_0,\frac{r'}{2^{j-i}}\right)\right|\\
        &\leq M \vertiii{u}_{k,p,\lambda} (r')^{\frac{\lambda - Q - p|I|_{\B{G}}}{p}} \sum_{s=0}^{j-i}2^{s\frac{Q+p|I|_{\B{G}}-\lambda}{p}}\\
        &= M \vertiii{u}_{k,p,\lambda} \frac{r^{\frac{\lambda - Q - p|I|_{\B{G}}}{p}}}{2^{i\frac{\lambda - Q - p|I|_{\B{G}}}{p}}} \sum_{s=0}^{j-i}2^{s\frac{Q+p|I|_{\B{G}}-\lambda}{p}}\\
        &= M \vertiii{u}_{k,p,\lambda} r'^{\frac{\lambda - Q - p|I|_{\B{G}}}{p}} \sum_{t=i}^{j}2^{t\frac{Q+p|I|_{\B{G}}-\lambda}{p}}.
    \end{split}
    \end{equation}
    Since $|I|_{\B{G}}\leq\alpha$ and $\lambda>Q+p\alpha>Q+p|I|_{\B{}G}$, the series $\sum 2^{t\frac{Q+p|I|_{\B{G}}-\lambda}{p}}$ converges. Hence, thanks to Cauchy's convergence test, \eqref{eqseries} ensures that the sequence $\{a_I\left(x_0,\frac{r}{2^h}\right)\}$ converges for $h\to\infty$. We can then define
    \begin{equation}\label{eqdefvI}
    v_I(x_0):=\lim_{h\to\infty} a_I\left(x_0,\frac{r}{2^h}\right).
    \end{equation}
    We now check that this definition does not depend on the choice of $r>0$. Indeed if we choose $0<r<R$ and apply Lemma \ref{degiorgicarnot} and the definition of the class $\C{L}^{p,\lambda}_k(\Omega)$ we get
    \begin{align}\label{eqindependentr}
        \left|a_I\left(x_0,\frac{r}{2^h}\right) \right. & \left. - a_I\left(x_0,\frac{R}{2^h}\right)\right|^p \nonumber\\
        \leq & C \frac{2^{h(Q+p|I|_{\B{G}})}}{r^{Q+p|I|_{\B{G}}}} \int_{\Omega(x_0,\frac{r}{2^h})}\left|P_k\left(x,x_0,\frac{r}{2^h}\right)-P_k\left(x,x_0,\frac{R}{2^h}\right)\right|^p dx \nonumber\\
        \leq & C \frac{2^{h(Q+p|I|_{\B{G}})}}{r^{Q+p|I|_{\B{G}}}} 2^{p-1} \left[\int_{\Omega(x_0,\frac{r}{2^h})}\left|P_k\left(x,x_0,\frac{r}{2^h}\right)-u(x)\right|^pdx \right. \\
        & + \left. \int_{\Omega(x_0,\frac{R}{2^h})}\left|P_k\left(x,x_0,\frac{R}{2^h}\right)-u(x)\right|^p dx \right] \nonumber\\
        \leq & C \frac{2^{h(Q+p|I|_{\B{G}})+p-1}}{r^{Q+p|I|_{\B{G}}}} \vertiii{u}^p_{k,p,\lambda} \frac{r^{\lambda}+R^{\lambda}}{2^{h\lambda}}\nonumber\\
        =& C 2^{p-1} \vertiii{u}^p_{k,p,\lambda} \frac{r^{\lambda}+R^{\lambda}}{r^{Q+p|I|_{\B{G}}}} 2^{h(Q+p|I|_{\B{G}}-\lambda)}.\nonumber
    \end{align}
    Thanks to the hypotheses $\lambda>Q+p\alpha>Q+p|I|_{\B{}G}$, this last quantity is infinitesimal for $h\to\infty$. Hence we have the desired independence from $r$.

    From Lemma \ref{lemma_changeradius} we have
    \begin{equation}\label{eqchangeradius}
    \left|a_I(x_0,r)-a_I\left(x_0,\frac{r}{2^h}\right)\right|\leq M \vertiii{u}_{k,p,\lambda} r^{\frac{\lambda-Q-p|I|_{\B{G}}}{p}} \sum_{j=0}^{h-1} 2^{j\frac{Q+p|I|_{\B{G}}-\lambda}{p}}.
    \end{equation}
    We already established the convergence of the series and in particular we can estimate the sum (from above) with a constant $\hat{M}$ that depends only on $Q,k,p$ and $\lambda$. Hence, by combining \eqref{eqdefvI} and \eqref{eqchangeradius} and by choosing the constant $N=M\hat{M}$ we conclude the proof.
\end{proof}

From the independence in the choice of $r$ in \eqref{eqdefvI} we get the following alternative definition of $v_I(x_0)$, that we shall use for the rest of the paper.

\begin{corollary}\label{coroll_unif}
Under the assumptions of Lemma \ref{lemmavI}, the following equality holds uniformly with respect to $x_0\in\overline{\Omega}$
\[
v_I(x_0)=\lim_{r\to 0} a_I(x_0,r).
\]
\end{corollary}

The functions $v_I$ will be the key player of Section \ref{Sec4}. From now we will omit the subscrip 0 and just write $v_I(x)$ for ease of notation.

\section{Regularity results}
\label{Sec4}

In this section we are going to study the regularity of the functions $v_I$ with respect to the choice of the parameters $k,p,\lambda$. We start with a result about the regularity of the functions $v_I$ when $|I|_{\B{G}}$ takes the highest value allowed by the space $\C{L}_k^{p,\lambda}(\Omega)$. 

\begin{proposition}\label{propholder}
    Suppose that $\Omega$ has the \eqref{thin} condition, $\lambda>Q+pk$, $u\in\C{L}_k^{p,\lambda}(\Omega)$ and $I$ is a multi-index with $|I|_{\B{G}}=k$. Then the function $v_I(x)$ is H\"older continuous in $\overline{\Omega}$ with exponent $\alpha=\frac{\lambda-Q-kp}{p}$ and in particular, given $x,y\in\overline{\Omega}$, we have the estimate
    \[
    |v_I(x)-v_I(y)|\leq \Theta \vertiii{u}_{k,p,\lambda} d(x,y)^{\alpha}
    \]
    where $\Theta$ is a constant that does not depend on $x,y$ and the choice of $I$.
\end{proposition}

\begin{proof}
    Let $I$ be a multi-index with $|I|_{\B{G}}=k$. Let us start by assuming that $x,y\in\overline{\Omega}$ are such that $r=d(x,y)\leq \frac{\diam{\Omega}}{2}$. We have that
    \begin{align}\label{eqstimavI}
        |v_I(x)-v_I(y)|&\leq |v_I(x) - a_I(x,2r)|\\
        &+|a_I(x,2r)-a_I(y,2r)|+|v_I(y)-a_I(y,2r)|.\nonumber
    \end{align}
    From Lemma \ref{lemmavI} we have that
    \begin{equation*}
    |v_I(x) - a_I(x,2r)|\leq N\vertiii{u}_{k,p,\lambda}(2r)^{\frac{\lambda-Q-pk}{p}}
    \end{equation*}
    and the same is true for $|v_I(y)-a_I(y,2r)|$. To estimate the missing part of 
    \eqref{eqstimavI} we use Lemma \ref{lemma_changebasepoint} to get
    \[
    |a_I(x,2r)-a_I(y,2r)|\leq C^{\frac{1}{p}}2^{\frac{p+\lambda}{p}}\vertiii{u}_{k,p,\lambda}r^{\frac{\lambda-Q-pk}{p}}.
    \]
    Hence \eqref{eqstimavI} becomes
    \[
    |v_I(x)-v_I(y)|\leq \vertiii{u}_{k,p,\lambda}r^{\frac{\lambda - Q - pk}{p}}2^{\frac{p+\lambda}{p}}\left( N 2^{-\frac{Q+pk}{p}} + C^{\frac{1}{p}}\right)
    \]
    and, by choosing the constant $\Theta = 2^{\frac{p+\lambda}{p}}\left( N 2^{-\frac{Q+pk}{p}} + C^{\frac{1}{p}}\right)$, the first part of the proof is concluded.

    To prove the proposition in the case $d(x,y)>\frac{\diam{\Omega}}{2}$, we first observe that $\Omega$ is connected. Hence, we can find points $x_0,\dots,x_m$ such that
    \begin{enumerate}[i)]
    \item $x_0=x$ and $x_m=y$;
    \item $x_i\in\overline{\Omega}$ for $i=1,\dots,m-1$;
    \item $|x_i^{-1}x_{i-1}|\leq\frac{\diam{\Omega}}{2}$ for $i=1,\dots,m.$
    \end{enumerate}
    The number of intermediate points needed can be bounded uniformly with respect to $x$ and $y$ by a constant that depends only $\Omega$. Hence we can apply the previous step to each couple $x_{i-1},x_i$ to conclude the proof.
\end{proof}

    We recall definition \ref{def_homog} of homogeneity of a coordinate in $\B{G}$.

\begin{theorem}\label{maintheorem}
    Suppose that $\Omega$ has the \eqref{thin} condition, $k\geq d_N$, $\lambda>Q+pk$, $u\in\C{L}_k^{p,\lambda}(\Omega)$ and $I$ is a multi-index with $|I|_{\B{G}}\leq k-d_N$. Then for each $x_0\in\Omega$ we have
    \begin{equation}\label{eq_derivative}
        X_i \left(v_I(x_0)\right) = v_{I+e_i}(x_0)\ \ \ \text{for each } i=1,\dots,N.
    \end{equation}
\end{theorem}

\begin{proof}
    \begin{align*}
        X_i(v_I(x_0)) &= X_i\left( \lim_{j\to\infty} a_I\left( x_0,\frac{r}{2^j} \right) \right) \stackrel{\dagger}{=} \lim_{j\to\infty} X_i\left( a_I \left( x_0,\frac{r}{2^j} \right) \right)\\ 
        & = \lim_{j\to\infty} \left[X_i\left( X^I \left(P_k\left(x,x_0,\frac{r}{2^j}\right)\right)\right)\right]_{x=x_0}\\
        & = \lim_{j\to\infty} \left[X^{I+e_i} \left(P_k\left(x,x_0,\frac{r}
        {2^j}\right)\right)\right]_{x=x_0}\\
        & = \lim_{j\to\infty} a_{I+e_i}\left( x_0,\frac{r}{2^j} \right) = v_{I+e_i}(x_0).
    \end{align*}
    Notice that the equality marked with $\dagger$ is possible because, thanks to Corollary \ref{coroll_unif}, the limit is uniform. The others follow directly from the definitions given so far.

\end{proof}

\begin{remark}
The limitations $k\geq d_N$ and $|I|_{\B{G}}\leq k-d_N$ are peculiar to the Carnot group setting. They are necessary because, in this setting, derivatives with respect to different coordinates can have different ``weights''.
    
Take for example the prototype case of the first Heisenberg group $\B{H}$, with coordinates $x,y,t$ with homogeneity respectively 1, 1 and 2. Consider a monomial of the form $x^a y^b t^c$, which has homogeneous degree $a+b+2c$, and take the derivative with respect to $x$ or $y$. The result is a monomial of homogeneous degree $a+b+2c-1$. However, if we take the derivative with respect to $t$, the result is a monomial of homogeneous degree $a+b+2c -2$. Hence, when giving the assumptions for Theorem \ref{maintheorem}, we need to make sure that $k$ is big enough and $|I|_{\B{G}}$ is small enough for all the derivatives to make sense.

The limitation imposed are still consistent with Campanato's work. Indeed, in the Euclidean case we have that all the coordinates have the same homogeneity and in particular $d_N=1$.
\end{remark}

\begin{remark}
    If we limit the range of $i$ in \eqref{eq_derivative} then we can relax the limitation imposed on $k$ (and, in turn, on $\lambda$ and $|I|_{\B{G}}$). For example, if we allow $i$ to range only between 1 and $m$ (recall that $m=m_1=\dim(V_1)$ is the dimension of the horizontal strata), then we shall go back to the original statement by Campanato \cite{Campanato64} with $k\geq 1$ and $|I|_{\B{G}}<k-1$.
\end{remark}

As a simple corollary of Proposition \ref{propholder} and Theorem \ref{maintheorem} we get the following

\begin{theorem}
\label{theochainderivative}
    Suppose that $\Omega$ has the \eqref{thin} condition, $k\geq d_N$, $\lambda>Q+pk$ and $u\in\C{L}_k^{p,\lambda}(\Omega)$. Then for each $x_0\in\Omega$, $v_{(0)}\in C^{k,\alpha}(\overline{\Omega})$ where $\alpha=\frac{\lambda-Q-pk}{p}$ and 
    \[
    X_I(v_{(0)}(x_0)) = v_I(x_0)\ \ \text{for all}\ \ I\ \text{with}\ |I|_{\B{G}}\leq k-d_N.
    \]
\end{theorem}
   
\begin{corollary}
     If $u\in\C{L}_k^{p,\lambda}(\Omega)$ with $\lambda>Q+p(k+1)$ then the functions $v_I(x)$ with $|I|_{\B{G}}=k$ are constants and the function $v_{(0)}(x)$ is a polynomial of homogeneous degree at most $k$.
\end{corollary}

\begin{proof}
    The second assertion comes directly from the first. The first assertion is a direct consequence of Proposition \ref{propholder} when $\lambda>Q+p(k+1)$, since in that case $\alpha$ becomes bigger than 1.
\end{proof}

To close this section we prove the equivalent of \cite[Theorem 5.1]{Campanato64}, which says that the regularity results proved so far for $v_{(0)}$ also hold for $u\in\C{L}^{p,\lambda}_k(\Omega)$.

\begin{theorem}
\label{Thm47}
    Suppose that $\Omega$ has the (\ref{thin}) condition and $u\in\C{L}^{p,\lambda}_k(\Omega)$ with $k\geq d_N$, $\lambda>Q+pk$. Then $u\in C^{k,\alpha}(\overline{\Omega})$ where $\alpha = \frac{\lambda - Q - pk}{p}$ and the following estimate holds:
    \begin{equation}\label{eq.holderestimate}
    [u]_{k,\alpha,\Omega}\leq \Theta \vertiii{u}_{k,p,\lambda}
    \end{equation}
    where $\Theta$ is the constant from Proposition \ref{propholder}.
\end{theorem}

\begin{proof}
    We follow the proof of \cite[Theorem 5.1]{Campanato64} and in particular we show that, under the assumptions of the theorem, $u$ coincides almost everywhere on $\Omega$ with $v_{(0)}$.

    Since $u\in L^p(\Omega)$, the following is true for almost every $x_0\in\Omega$
    \begin{equation}\label{eq_densitypoint}
        \lim_{\rho\to 0}\; \dashint_{\Omega(x_0,\rho)} |u(x)-u(x_0)|^pdx = 0.
    \end{equation}
    Take one of such points $x_0$. For almost any other $x\in\Omega$ we have that
    \begin{align*}
        |a_{(0)}(x_0,\rho) - u(x_0)|\leq& \, 3^{p-1}\left\{ |P_k(x,x_0,\rho) - a_{(0)}(x_0,\rho)|^p \right. \\
        & \left. |P_k(x,x_0,\rho) - u(x)|^p + |u(x)-u(x_0)|^p \right\}.
    \end{align*}
We now integrate on both side on $\Omega(x_0,\rho)$ to get
\begin{align}
        a_{(0)}(x_0,\rho) - u(x_0)|\leq\;&  \frac{3^{p-1}}{A\rho^Q}\int_{\Omega(x_0,\rho)}|P_k(x,x_0,\rho) - a_{(0)}(x_0,\rho)|^p dx \label{eq.part1}\\
        & + \frac{3^{p-1}}{A\rho^Q}\int_{\Omega(x_0,\rho)}|P_k(x,x_0,\rho) - u(x)|^p dx \label{eq.part2}\\
        & + \frac{3^{p-1}}{A\rho^Q}\int_{\Omega(x_0,\rho)}|u(x)-u(x_0)|^p dx. \label{eq.part3}
    \end{align}
    Observe that the integral in \eqref{eq.part2} is infinitesimal with $\rho$ due to the definition of $\C{L}^{p,\lambda}_k$ and so is the integral in \eqref{eq.part3}, thanks to \eqref{eq_densitypoint}. We are left to show that the integral in \eqref{eq.part1} is also infinitesimal with $\rho$. A simple computation yields
    \begin{align*}
    \int_{\Omega(x_0,\rho)}|P_k(x,x_0,\rho) - a_{(0)}(x_0,\rho)|^p dx &\leq \int_{\Omega(x_0,\rho)}\left|\sum_{1\leq |I|_{\B{G}}\leq k} a_{(I)}(x_0,\rho)x^I\right|^p dx \\
    &\leq \Lambda \sum_{1\leq |I|_{\B{G}}\leq k} |a_{(I)}(x_0,\rho)|^p\rho^{p|I|_{\B{G}}+Q},
    \end{align*}
    where $\Lambda$ is a constant that depends only on $k,p$ and $Q$. Hence also the right-hand side of \eqref{eq.part1} goes to zero with $\rho$, proving the assertion.
    
    The estimate \eqref{eq.holderestimate} now follows directly from \eqref{normCkalpha} and from Proposition \ref{propholder}. 
    \end{proof}
\begin{corollary}
    If, in the hypothesis of Theorem \ref{Thm47}, we further require $\lambda > Q + p(k+1)$, then $u$ coincide in $\Omega$ with a polynomial of homogeneous degree at most $k$.
\end{corollary}

\section{Conclusion}
\label{Sec5}
In this manuscript we have tabled a joint approach between the functional spaces firstly introduced by Morrey and then deeply studied by Campanato  and the theory of Carnot groups, which is becoming increasingly important for researchers in the field of analysis and mathematical physics. Whilst we have recalled the most significant definitions and classical results for both Morrey-Campanato spaces and Carnot groups, providing a useful introductory guide for those who begin to look out on these topics, we have also proved some properties concerning the class $\C{L}_k^{p,\lambda}(\Omega)$ and its regularity when $\Omega\subset\B{G}$. 

To conclude, we outline some possible further developments and research lines. On the one hand, the potential of Campanato's point of view could be applied in a more general setting, considering different Lie groups or even adapted to differentiable manifolds, under suitable hypotheses.  On the other hand, the increasing importance of Carnot groups with their applications, for instance, in control theory and quantum mechanics, should suggest a more specific study around the adjustment of other classical results of functional analysis in this framework. Finally, we want to emphasize that the original manuscript of Campanato, widely cited in this article, has other promising results\footnote{We refer here to the Section $6$ and Section $7$ in \cite{Campanato64}.} that can be reformulated in terms of Carnot groups, which we do not address in these pages lest we weigh down this initial analysis on this topic.
\section*{Statements and Declarations}

\noindent\textbf{Acknowledgements.} The authors are members and acknowledge the support of {\it Gruppo Nazionale per l’Analisi Matematica, la Probabilità e le loro Applicazioni} (GNAMPA) of {\it Istituto Nazionale di Alta Matematica} (INdAM).

Part of this work was done while M.~Capolli was a post-doctoral student at the institute of mathematics of the Polish Academy of Sciences (IMPAN) under the supervision of Dr. Hab. Prof. Tomasz Adamowich.

N.~Cangiotti acknowledges the support of the MIUR - PRIN 2017 project ``From Models to Decisions'' (Prot. N. 201743F9YE).

M. Capolli acknowledges the support granted by the European Union – NextGenerationEU Project “NewSRG - New directions in SubRiemannian Geometry” within the Program STARS@UNIPD 2021.



\bibliographystyle{plain}
\bibliography{main}

\end{document}